\documentclass[preprint,12pt]{elsarticle}
\usepackage[utf8]{inputenc}

\usepackage{fullpage}

\usepackage{amsmath}
\usepackage{amssymb}
\usepackage{latexsym}
\usepackage{amsthm}

\usepackage{tikz}

\usepackage{url}

\newtheorem{theorem}{Theorem}
\newtheorem{lemma}[theorem]{Lemma}
\newtheorem{corollary}[theorem]{Corollary}

\newtheorem{remark}[theorem]{Remark}
\newtheorem{conjecture}{Conjecture} 

\usepackage{minitoc}

\usepackage{ifpdf}

\journal{Discrete Mathematics}

\usepackage{lineno} 
\usepackage{enumitem} 

\newdefinition{problem}{Problem} 

\begin{document}

\begin{frontmatter}



\title{{Minimal obstructions to $(\infty, k)$-polarity in cographs\tnoteref{t1}}}

\tnotetext[t1]{This research was supported by the
SEP-CONACYT grant A1-S-8397, the CONACYT
FORDECYT-PRONACES/39570/2020 grant and the
DGAPA-PAPIIT grant IA104521.}

\author[FC]{F. Esteban Contreras-Mendoza}
\ead{esteban.contreras@ciencias.unam.mx}

\author[FC]{C\'esar Hern\'andez-Cruz\corref{cor1}}
\ead{chc@ciencias.unam.mx}

\address[FC]{Facultad de Ciencias\\
Universidad Nacional Aut\'onoma de M\'exico\\
Av. Universidad 3000, Circuito Exterior S/N\\
C.P. 04510, Ciudad Universitaria,
CDMX, M\'exico}

\cortext[cor1]{Corresponding author}

\begin{abstract}

A graph is a cograph if it does not contain a 4-vertex
path as an induced subgraph.   An $(s, k)$-polar
partition of a graph $G$ is a partition $(A, B)$ of its
vertex set such that $A$ induces a complete multipartite
graph with at most $s$ parts, and $B$ induces the
disjoint union of at most $k$ cliques with no other
edges.   A graph $G$ is said to be $(s, k)$-polar if it admits an $(s, k)$-polar partition.   The concepts of
$(s, \infty)$-, $(\infty, k)$-, and $(\infty,
\infty)$-polar graphs can be analogously defined.

Ekim, Mahadev and de Werra pioneered in the research on
polar cographs, obtaining forbidden induced subgraph
characterizations for $(\infty, \infty)$-polar cographs,
as well as for the union of $(\infty, 1)$- and
$(1, \infty)$-polar cographs.    Recently, a recursive
procedure for generating the list of cograph minimal
$(s,1)$-polar obstructions for any fixed integer $s$
was found, as well as the complete list of $(\infty,
1)$-polar obstructions.   In addition to these results,
complete lists of minimal $(s, k)$-polar cograph
obstructions are known only for the pair $(2, 2)$.

In this work we are concerned with the problem of
characterizing $(\infty, k)$-polar cographs for a
fixed $k$ through a finite family of forbidden induced
subgraphs. As our main result, we provide complete
lists of forbidden induced subgraphs for the cases
$k=2$ and $k=3$.   Additionally, we provide a partial
recursive construction for the general case.   By
considering graph complements, these results extend
to $(s, \infty)$-polar cographs.

\end{abstract}


\end{frontmatter}




\section{Introduction}\label{sec:Introduction}


All graphs in this paper are considered to be finite
and simple.   We refer the reader to \cite{bondy2008}
for basic terminology and notation.   In particular,
we use $P_k$ and $C_k$ to denote the path and cycle
on $k$ vertices, respectively.   We also use $\overline
G$ and $G[X]$ to denote the complement of a graph $G$
and the subgraph of $G$ induced by the vertex set $X$,
respectively.

Cographs have been independently defined by several
authors since the 1970's.   The original definition
of cograph was introduced by Corneil, Lerchs and
Stewart Burlingham in \cite{corneilDAM3}; it is based
on the following recursive criteria: $K_1$ is a cograph;
if $G$ is a cograph, then its complement $\overline G$
is also a cograph; if $G$ and $H$ are cographs, so is
their disjoint union.   One of the best known
characterizations of cographs is that they
are precisely the $P_4$-free graphs (graphs without
$P_4$ as an induced subgraph).

Alternatively, cographs can be defined as graphs that
can be constructed from single vertex graphs by means
of disjoint-union and join operations.   In
\cite{corneilDAM3} a special rooted tree was introduced
to represent a cograph $G$: the leaf vertices are
associated with the vertices of $G$, and each internal
node is labeled $0$ or $1$ indicating the operation,
join or disjoint-union, performed on the cographs
associated with their children, respectively.
Furthermore, this tree must be such that the nodes in
a root-leaf path have alternating labels for ensuring
that each cograph is associated with only one of these
trees. Such a tree is called the {\em cotree} associated
with $G$.

A graph property is said to be {\em hereditary} if every
induced subgraph of a graph with such property also has
the property.   In \cite{damaschkeJGT14}, Peter Damaschke
proved that every hereditary property can be
characterized by a finite family of forbidden induced
subgraphs when restricted to cographs.   Thus, finding
the family of minimal forbidden induced subgraphs
characterizing a given hereditary property in the
class of cographs comes as a natural problem.

A {\em cluster} is the complement of a complete
multipartite graph, and, for a non-negative integer
$k$, a {\em $k$-cluster} is a cluster with at
most $k$ components. Given non-negative integers
$s$ and $k$, we define an {\em $(s,k)$-polar
partition} of a graph $G$ to be a partition
$(A,B)$ of $V_G$ such that $A$ induces a
complete $s$-partite graph and $B$ induces a
$k$-cluster. If a graph admits an $(s,k)$-polar
partition, we will say that it is {\em $(s,
k)$-polar}. We will use $\infty$ instead of $s$,
$k$ or both, to indicate that the number of parts
in the multipartite graph, or the number of
components in the cluster, respectively, is
unbounded.   Hence, we will say that a graph
$G$ is an $(s, \infty)$-polar graph if its
vertex set admits a partition $(A,B)$ where
$A$ is a complete $s$-partite graph, and $B$
is a cluster; such partition is an
$(s, \infty)$-polar partition.   The concepts
of $(\infty, k)$- and $(\infty, \infty)$-polar
graphs and partitions are analogously defined.
A {\em polar graph} is just an $(\infty,
\infty)$-polar graph.

Clearly, for $s,k \in \mathbb Z^+ \cup \{0,
\infty\}$, having an $(s,k)$-polar partition is
a hereditary property, and thus, as we have
already mentioned, $(s,k)$-polar cographs can be
characterized by a finite family of forbidden
induced subgraphs\footnote{This observation also
follows from the results in \cite{feder2006}, in
the context of matrix partitions of cographs.}.
A {\em cograph $(s,k)$-polar obstruction} is a
cograph which is not $(s,k)$-polar, and a {\em
cograph minimal $(s,k)$-polar obstruction} is a
cograph $(s, k)$-polar obstruction such that
every proper induced subgraph is $(s,k)$-polar.


Chernyak and Chernyak proved in
\cite{chernyakDM62} that determining whether
a graph is polar is an $\mathcal{NP}$-complete
problem, in \cite{farrugiaEJC11}
Farrugia proved that the problem remains
$\mathcal{NP}$-complete for
$(1, \infty)$-polar graphs, and Churchley and
Huang proved in \cite{churchleyJGT76} that
the latter problem remains
$\mathcal{NP}$-complete even when restricted
to triangle-free graphs. In contrast, the
results on sparse-dense partitions in
\cite{federSIAMJDM16} imply that for any fixed
intergers $s$ and $k$, $(s,k)$-polar graphs
can be recognized in polynomial time. Recognition
and minimal obstructions for polar graphs have been
studied for other graph classes, e.g.,
chordal graphs \cite{ekimDAM156}, permutation
graphs \cite{ekimIWOCA2009}, or graphs having
a polar line graph \cite{churchleySIAMJDM25}.
In \cite{leTCS528}, several other families are
studied for the complexity of the recognition
problem for polar and $(1, \infty)$-polar
graphs; in particular, families having a
polynomial $(1, \infty)$-polar recognition
problem together with a small subfamily having
an $\mathcal{NP}$-complete polar recognition
problem are presented.

In terms of minimal obstructions, for very small
values of $s$ and $k$ the minimal $(s,k)$-polar
obstructions are well known; a graph is
$(0,k)$-polar if and only if it is a disjoint
union of at most $k$-cliques, it is $(s,0)$-polar
if and only if it is a complete $s$-partite graph,
and it is $(1,1)$-polar if and only if it is a
split graph.   It was shown by Foldes and Hammer
\cite{foldesSECGTC} that a graph is split if and
only if it is $\{ 2K_2, C_4, C_5 \}$-free; it is
folklore that a graph is a disjoint union of at
most $k$-cliques if and only if its independence
number is at most $k$ and it is $P_3$-free, which
by complementation implies that a graph is a
complete $s$-partite graph if and only if it is
$\{K_{s+1}, \overline{P_3}\}$-free.

In this work we focus on cograph minimal
obstructions for $(s,k)$-polarity. Ekim, Mahadev
and de Werra proved in \cite{ekimDAM156a} that
there are only eight cograph minimal polar
obstructions, and sixteen cograph minimal
$(s,k)$-polar obstructions when $\min\{ s, k \}
= 1$, \cite{ekimDAM171}.  Hell, Hern\'andez-Cruz
and Linhares-Sales proved in \cite{hellDAM},
that there are $48$ cograph minimal $(2,2)$-polar
obstructions. The exhaustive list of nine cograph
minimal $(2,1)$-polar obstructions was found by
Bravo, Nogueira, Protti and Vianna, \cite{bravo}.
Recently, Contreras-Mendoza and Hern\'andez-Cruz
exhibited a simple recursive characterization to
obtain all the cograph minimal $(s, 1)$-polar
obstructions for an arbitrary integer $s$, as
well as the complete list of cograph minimal
$(\infty, 1)$-polar obstructions \cite{contrerasDAM}.

We provide a partial recursive characterization for
cograph minimal $(\infty, k)$-polar obstructions. We
also exhibit complete lists of cograph minimal
$(\infty, k)$-polar obstructions for $k = 2$ and
$k = 3$.   By taking complements it is trivial to
obtain analogous results for $(s, \infty)$-polar
cographs.

We say that a component of $G$ is {\em trivial}
or an {\em isolated vertex} if it is isomorphic
to $K_1$. Given graphs $G$ and $H$, the disjoint
union of $G$ and $H$ is denoted by $G + H$, and
the join of $G$ and $H$ is denoted by $G \oplus
H$.   Thus, the sum of $n$ disjoint copies of
$G$ is denoted by $nG$.

Let $k,c,i$ be integers such that $0 \le i < c
\le k+2$.   We say that a cograph minimal $(\infty,
k)$-polar obstruction has {\em type} $(c,i)$ if it
has exactly $c$ connected components and precisely
$i$ of them are trivial. We divide the types in
three main classes: type $(c,0)$, which corresponds
to obstructions without isolated vertices, type
$(c,c-1)$, associated with obstructions
that have precisely one non-trivial component,
and the rest of the types, that contain obstructions
with at least one isolated vertex and at least one
complete component of order $2$.

The rest of the paper is organized as follows.
In Section \ref{sec:PreliminaryResults}, we
introduce some previously known results and
use them to characterize connected cograph
minimal $(\infty,k)$-polar obstructions.
Sections \ref{sec:type(c,0)} and \ref{sec:type(c,c-1)}
are devoted to the study of the structure
of $(c,0)$- and $(c,c-1)$-type obstructions,
respectively; a recursive characterization for the
obstructions of type $(c,0)$ is given.
In Section \ref{sec:remainingTypes} we establish
a pleasant simple characterization for the rest of
the cograph minimal $(\infty, k)$-polar obstructions.
Finally, in Section \ref{sec:MainResults} we
prove our main results, we exhibit complete
lists of cograph minimal $(\infty, k)$-polar
obstructions for the cases $k=2$ and $k=3$.
Conclusions and future lines of work are
presented in Section \ref{sec:Conclusions}.


\section{Preliminary results}
\label{sec:PreliminaryResults}


We begin this section introducing three
previously-known results concerning the
characterizations of $(1, \infty)$-,
$(\infty, \infty)$-, and $(1, s)$-polar
cographs for a fixed positive integer $s$.
Such results will be helpful in the
development of various of our results.

\begin{theorem}[Contreras-Mendoza \&
Hern\'andez-Cruz, \cite{contrerasDAM}]
\label{theo:Charact(1,infty)PO}
A graph $G$ is a cograph minimal $(1,
\infty)$-polar obstruction if and only if
$G$ is isomorphic to one of the graphs
depicted in Figure \ref{fig:CM(1,infty)PO}.
\end{theorem}

\begin{figure}[h!]
\begin{center}
\begin{tikzpicture}
[every circle node/.style ={circle,draw,minimum size= 5pt,
inner sep=0pt, outer sep=0pt},
every rectangle node/.style ={}];

\begin{scope}[xshift=-10, scale=1.5]
\node [circle] (1) at (0,0)[]{};
\node [circle] (2) at (1,1)[]{};
\node [circle] (3) at (0,1)[]{};
\node [circle] (4) at (1,0)[]{};
\node [circle] (5) at (0.5,0.5)[]{};
\foreach \from/\to in {1/3,1/4,1/5,2/3,2/4,2/5,3/5,4/5}
\draw [-, shorten <=1pt, shorten >=1pt, >=stealth, line width=.7pt]
(\from) to (\to);
\node [rectangle] (1) at (0.5,-0.625){$K_1 \oplus C_4$};
\end{scope}

\begin{scope}[xshift=140, scale=1.25, rotate=90]
\node [circle] (1) at (0,0)[]{};
\node [circle] (2) at (1,0)[]{};
\node [circle] (3) at (0,1)[]{};
\node [circle] (4) at (1,1)[]{};
\node [circle] (5) at (0,2)[]{};
\node [circle] (6) at (1,2)[]{};
\foreach \from/\to in {1/2,1/3,1/4,2/4,2/3,3/4,3/4,3/5,3/6,4/5,4/6,5/6}
\draw [-, shorten <=1pt, shorten >=1pt, >=stealth, line width=.7pt]
(\from) to (\to);
\node [rectangle] (1) at (-0.75,1){$K_2 \oplus 2K_2$};
\end{scope}

\begin{scope}[xshift=180, scale=1.5]
\node [circle] (1) at (0,0)[]{};
\node [circle] (2) at (1,0)[]{};
\node [circle] (3) at (1,1)[]{};
\node [circle] (4) at (0,1)[]{};
\node [circle] (5) at (0.3,0.6)[]{};
\node [circle] (6) at (0.7,0.6)[]{};
\foreach \from/\to in {1/2,1/4,1/5,1/6,2/3,2/5,2/6,3/4,3/6,4/5,5/6}
\draw [-, shorten <=1pt, shorten >=1pt, >=stealth, line width=.7pt]
(\from) to (\to);
\node [rectangle] (1) at (0.5,-0.625){$\overline{2P_3}$};
\end{scope}

\begin{scope}[xshift=270, scale=1.1]
\node [circle] (1) at (1,1)[]{};
\node [circle] (2) at (0,0)[]{};
\node [circle] (3) at (0,1)[]{};
\node [circle] (4) at (0,2)[]{};
\node [circle] (5) at (1.8,1.5)[]{};
\node [circle] (6) at (1.8,0.5)[]{};
\foreach \from/\to in {1/2,1/3,1/4,1/5,1/6,2/3,3/4,5/6}
\draw [-, shorten <=1pt, shorten >=1pt, >=stealth, line width=.7pt]
(\from) to (\to);
\node [rectangle] (1) at (0.9,-0.852){$K_1 \oplus (K_2+ P_3)$};
\end{scope}

\end{tikzpicture}
\end{center}
\caption{$(1, \infty)$-polar obstructions.}
\label{fig:CM(1,infty)PO}
\end{figure}
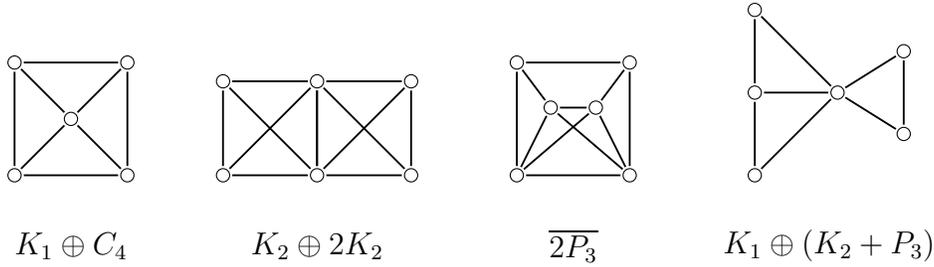

\begin{theorem}[Ekim, Mahadev \& de Werra,
\cite{ekimDAM156a}] 
\label{theo:CharactPolarCographs}
A graph $G$ is a cograph minimal polar
obstruction if and only if $G$ or its
complement is isomorphic to $P_3 + H$,
where $H$ is any cograph minimal $(1,
\infty)$-polar obstruction.
\end{theorem}

\begin{theorem}[Contreras-Mendoza \&
Hern\'andez-Cruz, \cite{contrerasDAM}] 
\label{theo:Charact(1,k)PO}
Let $s$ be an integer, $s \ge 2$.
\begin{enumerate}
	\item The graph $G$ is a connected cograph
		minimal $(1, s)$-polar obstruction if and
		only if $G$ is either a cograph minimal
		$(1, \infty)$-polar obstruction or it is
		isomorphic to $K_{s+1, s+1}, \overline{K_2}
		\oplus (K_2 + sK_1)$, or $K_1 \oplus (2K_2
		+ (s-1)K_1)$.

	\item The graph $G$ is a disconnected cograph
		minimal $(1, s)$-polar obstruction if and
		only there exists a positive integer $t$ and
		non-negative integers $s_0, s_1, \dots, s_t$
		such that $G = G_0 + \dots + G_t$, where $G_i$
		is a connected cograph minimal $(1, s_i)$-polar
		obstruction that is not a cograph minimal $(1,
		\infty)$-polar obstruction, and $s = t +
		\sum_{i=0}^t s_i$.
\end{enumerate}
\end{theorem}

Note that if $G$ is a cluster such that, either
$G$ has at most $k+1$ components, or $G$ has at
most $k$ non-trivial components, then $G$ is an
$(\infty,k)$-polar graph.   Hence, every cograph
minimal $(\infty,k)$-polar obstruction that is a
cluster, has at least $k+2$ components and at
least $k+1$ of them are non-trivial. In consequence,
we have the following useful observation.

\begin{remark} 
\label{rem:OnlyCluster(Infty,k)PO}
Let $k$ be an integer.   Up to isomorphism, the
graph $K_1 + (k+1)K_2$ is the only cograph minimal
$(\infty,k)$-polar obstruction that is a cluster.
\end{remark}

The following lemma is a slight modification of
Lemma 1 in \cite{hellDAM}; the proof is very
similar, and thus will be omitted.

\begin{lemma} 
\label{lem:BasicsOf(Infty,k)PO}
Let $k$ be a non-negative integer, and let $G$ be a cograph
minimal $(\infty, k)$-polar obstruction.   Then
\begin{enumerate}
	\item \label{it:AtMostk+2Components} $G$ has at most
		$k+2$ connected components,
	\item \label{it:AtLestOneNonTrivialComponent} $G$ has
		at least one non-trivial component,
	\item \label{it:AtMostk+TrivialComponent} $G$ has at
		most $k+1$ trivial components,
	\item \label{it:AtMostOneNoncompleteComponent} if $G$
		has at least one trivial component, then $G$ has at
		most one non-complete component,
	\item \label{it:OrderCompleteComponents} every complete
		component of $G$ has order one or two.
\end{enumerate}
\end{lemma}

We finish this section by characterizing connected
cograph minimal $(\infty,k)$-polar obstructions
for every integer $k$ such that $k \ge 2$.
The $(\infty, 0)$-polar cographs are precisely
the complete multipartite graphs, and it is well
known that the only cograph minimal
$(\infty,0)$-polar obstruction is $\overline{P_3}$.
Furthermore, from Theorem \ref{theo:Charact(1,infty)PO},
all the cograph minimal $(\infty,1)$-polar obstructions
are disconnected.   Thus, for $k \le 1$, there exist no
connected cograph minimal $(\infty, k)$-polar
obstructions.   In contrast, for $k \ge 2$ there
such minimal connected obstructions exist, and there
is a fixed number of them as we show below.

\begin{theorem} 
\label{theo:ConnectedObsmin}
Let $k$ be an integer, $k \ge 2$, and let $G$ be a
disconnected graph.   Then, $G$ is a cograph minimal
$(k, \infty)$-polar obstruction if and only if $G$ is a
cograph minimal polar obstruction.

Equivalently, $G$ is a connected cograph minimal
$(\infty, k)$-polar obstruction if and only if $G$ is a
connected cograph minimal polar obstruction.
\end{theorem}

\begin{proof}
Suppose that $G$ is a disconnected cograph minimal
$(k, \infty)$-polar obstruction.   Note that, by the
minimality of $G$, if $K$ is a complete component of
$G$, then $G-K$ is a $(k,\infty)$-polar graph, and thus
also is $G$, which is absurd. Hence, every component
of $G$ is non-complete.   Moreover, since $G$ is
not a $(1, \infty)$-polar graph, then $G$ contains a
cograph minimal $(1, \infty)$-polar obstruction $G'$
as an induced subgraph.   From Theorem
\ref{theo:Charact(1,infty)PO}, $G'$ is connected, so it
is completely contained in a single component of
$G$.   Thus, since $G$ has no complete components, $G$
contains $G' + P_3$ as an induced subgraph, but by
Theorem \ref{theo:CharactPolarCographs}, $G' + P_3$ is
a not $(k, \infty)$-polar, so $G = G' + P_3$,
which proves that $G$ is a cograph minimal polar
obstruction.   The converse implication follows easily
from Theorems \ref{theo:Charact(1,infty)PO} and
\ref{theo:CharactPolarCographs}.   The second statement
is an immediate consequence, because a graph $G$ is a
connected cograph minimal $(s, k)$-polar obstruction if
and only if its complement, $\overline G$, is a
disconnected cograph minimal $(k, s)$-polar
obstruction.
\end{proof}

\section{Type $(c,0)$ obstructions}
\label{sec:type(c,0)}

Since we have already characterized the connected
cograph minimal $(\infty,k)$-polar obstructions,
we are now only concerned within the disconnected
obstructions. Recall that the cograph minimal
obstructions of type $(c,0)$ are those without
isolated vertices. We begin our study of this
type noticing some restrictions on their connected
components.

\begin{lemma} 
\label{lem:Geq2NonCompleteComponents}
Let $k$ be an integer, $k \ge 2$, and let $G$ be a
disconnected cograph minimal $(\infty, k)$-polar
obstruction without isolated vertices.   Then, $G$
has at least two non-complete components.
\end{lemma}

\begin{proof}
Let $G$ be as in the hypothesis.   From Remark
\ref{rem:OnlyCluster(Infty,k)PO}, we have that $G$ is
not a cluster, so $G$ has at least one non-complete
component.   Suppose for a contradiction that $G$ has
precisely one non-complete component.   Then, by
Lemma \ref{lem:BasicsOf(Infty,k)PO}, for some integer
$j\in \{1, \dots ,k+1\}$, $G \cong jK_2 + H$, where $H$
is a connected non-complete graph.   Note that since $G$
is not an $(\infty, k)$-polar graph, $H$ is not an
$(\infty, k-j)$-polar graph.

Let $v$ be a vertex of $H$, and suppose that $H-v$ is
not a cluster.   Thus, for every $(\infty, k)$-polar
partition $(A, B)$ of $G-v$,
$A \cap V_{H-v} \neq \varnothing$, which implies that
$(A \cap V_{H-v}, B \cap V_{H-v})$ is an
$(\infty, k-j)$-polar partition of $H-v$.   Hence, for
each vertex $v$ of $H$, $H-v$ is either a cluster or an
$(\infty, k-j)$-polar graph.

Since $H$ is not an $(\infty, k-j)$-polar graph, $H$
contains a cograph minimal $(\infty, k-j)$-polar
obstruction $H'$ as an induced subgraph.   Nevertheless,
by Theorem	\ref{theo:ConnectedObsmin}, if $H'$ is
connected, then it is a cograph minimal
$(\infty, k)$-polar obstruction, in contradiction with
the minimality of $G$.   Thus $H'$ is a disconnected
induced subgraph of the connected cograph $H$.   Let $v$
be a vertex of $H - H'$.   Since $H'$ is an induced
subgraph of $H-v$, we have that $H-v$ is not an $(\infty,
k-j)$-polar graph, which implies that $H-v$ and $H'$
are clusters. However, from Remark
\ref{rem:OnlyCluster(Infty,k)PO}, $H'$ is isomorphic
to $K_1 + (k-j+1)K_2$, but in such a case $G$ properly
contains $K_1 + (k+1)K_2$ as an induced subgraph,
contradicting its minimality. The contradiction arose
from supposing that $G$ has no more than one
non-complete component, so  $G$ must have at least two
non-complete components.
\end{proof}

The following lemma characterizes a family of graphs with
some properties that are common to all cograph minimal
$(\infty, k)$-polar obstructions without isolated vertices.
It will be very useful for giving recursive constructions
of such obstructions.

\begin{lemma} 
\label{lem:CharacterizationOfAitchs}
Let $k$ be a positive integer, and let $H$ be a cograph.
Then, $H$ is such that
\begin{enumerate}
	\item \label{it:HisNotAcluster} $H$ is not a cluster,
	\item \label{it:1,k-1Obstruction} $H$ is
		$(1, k)$-polar, but not $(1, k-1)$-polar, and
	\item \label{it:ForEachVertexVofH} for each vertex $v$
		of $H$, the graph $H-v$ is either a $(1, k-1)$-polar
		graph or a cluster,
\end{enumerate}
if and only if exactly one of the following statements
is satisfied:
\begin{enumerate}[label={\alph*.}]
	\item \label{it:1,k-1MinimalObstruction} $H$ is a
		cograph minimal $(1, k-1)$-polar
		obstruction, that is neither a cograph minimal
		$(1, \infty)$-polar obstruction nor isomorphic to
		$kK_2$.
	\item \label{it:P3+(k-1)K2} $H \cong P_3 + (k-1)K_2$.
	\item \label{it:(k-2)K2+(K1Oplus2K2)} $k \ge 2$, and
		$H \cong (k-2)K_2 + (K_1 \oplus 2K_2)$.
\end{enumerate}
\end{lemma}

\begin{proof}

Let $H$ be a cograph that satisfies items
\ref{it:HisNotAcluster} to \ref{it:ForEachVertexVofH}.
Since $H$ is a cograph, we have from item
\ref{it:1,k-1Obstruction} that $H$ contains a cograph
minimal $(1, k-1)$-polar obstruction $H'$.   Observe
that, since $H$ is not a cluster but it is
$(1,k)$-polar, if $H = H'$ then it satisfies item
{\em a.}

Let $K$ be a complete component of $H$ (if any). We
claim that $K$ has order two.   From item 3, for every
vertex $w$ of $K$, $H-w$ admits a $(1, k-1)$-polar
partition $(A, B)$.   If $K$ is a trivial graph, then
$(A \cup \{w\}, B)$ is a $(1, k-1)$-polar partition of
$H$, which is impossible.   Else, if $K$ has order at
least three, then $V_{K-w} \cap B \neq\varnothing$
(otherwise $A = V_{K-w}$ and $B$ covers $H-K$, which
cannot occur since $H-K$ is not a cluster), and then
$(A, B \cup \{w\})$ is a $(1, k-1)$-polar partition of
$H$, a contradiction.   Therefore, every complete
component of $H$ is isomorphic to $K_2$.

Suppose that $H$ properly contains a
cograph minimal $(1, k-1)$-polar obstruction as an
induced subgraph.   It implies that there exists a
vertex $v$ of $H$ such that $H-v$ is not a
$(1, k-1)$-polar graph, and from item 3, $H-v$ is a
cluster.   Note that from item 1, $H$ has a subgraph $P$
isomorphic to $P_3$, and that $v$ is necessarily a
vertex of $P$, or $H-v$ would not be a cluster.

Let $v$ and $P$ be as described above, then we have two
cases: either $d_P(v)=1$ or $d_P(v)=2$.   Suppose first
that $d_P(v) = 1$.   Since $H$ is a
cograph and $d_P(v) = 1$, $v$ is adjacent to exactly one
component of the cluster $H-v$, and therefore
$H \cong jK_2 + (K_a \oplus (v + K_b))$ for some
positive integers $a$ and $b$ and some non-negative
integer $j$.   Moreover, since $H$ is not a
$(1, k-1)$-polar graph, $j \ge k-1$, but if $j > k-1$
then $H$ contains $(k+1)K_2$ as a proper induced
subgraph, and then it is not a $(1, k)$-polar graph,
contradicting our assumptions.   Thus, $j= k-1$, and
$H \cong (k-1)K_2 + (K_a \oplus (v + K_b))$.   Observe
that $H$ contains $H' \cong P_3 + (k-1)K_2$ as an
induced subgraph, and $H'$ is neither $(1,k-1)$-polar
nor a cluster.   Therefore, by item 3,
$H \cong P_3 + (k-1)K_2$, that is, $H$ satisfies item
{\em b.}

For the second case, suppose that $d_P(v) = 2$.   Note
that since $v$ is adjacent to at least two components of
the cluster $H-v$, then $v$ is completely adjacent or
completely not-adjacent to each component of $H-v$, and
therefore, it is completely adjacent to at least two
components of $H-v$.   Let $K$ be a component of $H-v$
that is completely adjacent to $v$, and suppose for a
contradiction that $K$ has more than two vertices: if
$w$ is a vertex of $K$, then $H-w$ is not a cluster, so
it admits a $(1, k-1)$-polar partition $(A,B)$ and
therefore $(A, B \cup \{w\})$ is a $(1, k)$-polar
partition of $H$, a contradiction.   Hence, every
component of $H-v$ that is completely adjacent to $v$
has at most two vertices, and in consequence $H$ is
isomorphic to $qK_2 + (v\oplus (\ell K_2 + mK_1))$ for
some non-negative integers $\ell, m$ and $q$ such that
$\ell + m \ge 2$.

Observe that if $\ell + q \ge k+1$ then $H$ contains
$(k+1)K_2$ as an induced subgraph, and then $H$ is not a
$(1, k)$-polar cograph, contradicting our hypothesis.
Therefore, $\ell + q \le k$.   Furthermore, since $H-v$
is a cluster that is not a $(1, k-1)$-polar graph, it
contains a cograph minimal $(1, k-1)$-polar obstruction
$H'$ that is a cluster as an induced subgraph.
Nevertheless, the only cograph minimal $(1, k-1)$-polar
obstruction that is a cluster is $H' \cong kK_2$.
The above observation implies that $\ell + q \ge k$,
so we have that $\ell + q = k$.

It is straightforward to show that if $\ell \le 1$, then
$H$ has $P_3 + (k-1)K_2$ as a proper induced subgraph,
which is impossible as we have noted
when proving the case $d_P(v)=1$.   Thus,
$\ell \ge 2$.   Furthermore, note that the component of
$H$ that contains $v$ is a $(1, \ell + m)$-polar graph
that is not a $(1, \ell + m -1)$-polar graph, which
implies that $H$ is a $(1, k + m)$-polar graph that
admits no $(1, k+m-1)$-polar partitions.   However, by
hypothesis $H$ is a $(1,k)$-polar graph that is not a
$(1, k-1)$-polar graph, so we have that $m = 0$, and
then $H \cong (k-\ell)K_2 + (v \oplus \ell K_2)$.
Suppose for a contradiction that $\ell \ge 3$, and
let $w$ be a vertex of $H$ adjacent to $v$.   Since
$H-w$ is not a cluster, it is a $(1, k-1)$-polar graph,
and therefore the component of $H-w$ that contains $v$
is $(1, \ell -1)$-polar, but this is impossible
since such component is isomorphic to $K_1 \oplus
((\ell -1)K_2 + K_1)$ which contains the cograph
minimal $(1, \ell-1)$-polar obstruction $K_1 \oplus
(2K_2 + (\ell-2)K_1)$ as an induced subgraph.
Hence, $\ell =2$ and $H \cong (k-2)K_2 + (K_1
\oplus 2K_2)$, so item {\em c.} is satisfied.

To prove that the graphs described in items {\em b.} and
{\em c.} satisfy the sentences of items 1, 2, and 3 is a
simple routine work.   The analogous result for graphs
described in item {\em a.} follows from Theorem
\ref{theo:Charact(1,k)PO}.
\end{proof}

It also will be useful to know when do the graphs
described in the above lemma posses some specific
properties. The following remark identifies some
interesting cases.   The proof is straightforward
and thus omitted.

\begin{remark} 
\label{rem:(Infty,k-1)PolarityAndKClusters}
Let $k$ be an integer and, let $H$ be a cograph.
\begin{enumerate}
	\item Suppose that $H$ is a cograph minimal
		$(1, k-1)$-polar obstruction that is neither a
		cograph minimal $(1, \infty)$-polar obstruction nor
		isomorphic to $kK_2$.   Then, $H$ is an
		$(\infty, k-1)$-polar graph if and only if $H$ has
		precisely one component non-isomorphic to $K_2$.
	\item The graph $H$, with $H \cong P_3 + (k-1)K_2$ is a
		$(2, k-1)$-polar graph, and for each vertex $v$
		of $H$, $H -v$ is either a $(1, k-1)$-polar graph
		or it is isomorphic to $kK_2$.
	\item The graph $H$, with $H \cong (k-2)K_2 + (K_1
		\oplus 2K_2)$ is a $(3, k-1)$-polar graph, and for
		each vertex $v$ of $H$, $H-v$ is either a $(1,
		k-1)$-polar graph or it is isomorphic to $kK_2$.
\end{enumerate}
\end{remark}

It results convenient to divide the study of
disconnected cograph minimal $(\infty, k)$-polar
obstructions without isolated vertices into two
cases, depending on whether some component is
isomorphic to $P_3$.   We start by treating the
case in which the graphs have not components
isomorphic to $P_3$.

\begin{lemma} 
\label{lem:ObsminWhitoutP3CompNorIsolated}
Let $k$ be a non-negative integer, and let $G$ be a
graph without components isomorphic to $P_3$.   Then,
$G$ is a disconnected cograph minimal $(\infty,
k)$-polar obstruction without isolated vertices if
and only if there exist positive integers $k_1$ and
$k_2$, and cographs $H_1$ and $H_2$ such that $G =
H_1 + H_2$, and for $i\in\{1,2\}$, the following
statements are satisfied:
\begin{enumerate}
	\item $H_i$ is not a cluster,
	\item $H_i$ is a $(1, k_i)$-polar graph that admits no
		$(1, k_i-1)$-polar partitions,
	\item for each vertex $v$ of $H_i$, the graph $H_i -v$
		is either a $(1, k_i-1)$-polar graph or a cluster,
	\item for $j\in\{1,2\}$ such that $j \neq i$, if $H_i$
		is not a cograph minimal $(1, k_i-1)$-polar
		obstruction, then $H_j$ is an
		$(\infty ,k_j -1)$-polar graph, and
	\item $k = k_1 + k_2 -1$.
\end{enumerate}
\end{lemma}

\begin{proof}
Suppose that $G$ is a disconnected cograph minimal
$(\infty, k)$-polar obstruction without isolated
vertices.   From Lemma
\ref{lem:Geq2NonCompleteComponents}, $G$ has at least
two non-complete components, say $G_1$ and $G_2$.
Let $H_1 = G_1$ and $H_2 = G - G_1$.   Evidently,
$G = H_1 + H_2$ and both, $H_1$ and $H_2$, are cographs
that are not clusters.

Let $\{i, j\} = \{1,2\}$. Observe that, since $G_j$ is a
non-complete component of $G$ and $G$ has no components
isomorphic to $P_3$, there exists a vertex $v$ of $H_j$
such that $H_j - v$ is not a cluster.   In addition, by
the minimality of $G$, $G-v$ admits an
$(\infty, k)$-polar partition $(A,B)$, but
$G-v = H_i + (H_j-v)$, so $G-v$ has at least two
non-complete components, and therefore $(A,B)$ is a
$(1, k)$-polar partition.   Furthermore, since $H_i$
contains $P_3$ as an induced subgraph, it is not a
$(1,0)$-polar graph.   The above observations imply
that there exists an integer $k_i\in\{1, \dots, k-1\}$
such that $H_i$ is a $(1, k_i)$-polar graph that is
not $(1, k_i -1)$-polar.   Note that $G$ is a
$(1, k_1+k_2)$-polar graph that is not
$(\infty, k)$-polar, which implies that
$k \le k_1 + k_2 -1$.

Let $v$ be a vertex of $H_i$, and let $(A, B)$ be an
$(\infty, k)$-polar partition of $G-v$.   If $H_i-v$
is not a cluster, and given that $H_j$ is neither,
$(A,B)$ is a $(1, k)$-polar partition, and since $H_j$
is not a $(1, k_j -1)$-polar graph, then
$(A \cap V_{H_i-v}, B\cap V_{H_i-v})$ is a
$(1, k - k_j)$-polar partition,
which implies that $H_i -v$
is a $(1, k_i-1)$-polar graph, because
$k - k_j \le k_i -1$.   Therefore, for each vertex $v$
of $H_i$, the graph $H_i -v$ is either a cluster or a
$(1, k_i -1)$-polar graph.

Suppose that $H_i$ is not a cograph minimal
$(1, k_i -1)$-polar obstruction.   Since $H_i$ is not a
$(1, k_i-1)$-polar graph, it follows from Lemma
\ref{lem:CharacterizationOfAitchs} and Remark
\ref{rem:(Infty,k-1)PolarityAndKClusters} that there
exists a vertex $v$ of $H_i$ for which
$H_i -v \cong k_iK_2$.   Let $(A, B)$ be an
$(\infty, k)$-polar partition of $G-v$. The graph $H_j$
is not a cluster, so we have that
$A \cap V_{H_j} \neq \varnothing$, and then
$(A \cap V_{H_j}, B \cap V_{H_j})$ is an
$(\infty, k-k_i)$-polar partition of $H_j$, and
therefore $H_j$ is an $(\infty, k_j-1)$-polar graph,
because $k-k_i \le k_j-1$.   Hence, if $H_i$ is not a
cograph minimal $(1, k_i-1)$-polar obstruction, then
$H_j$ is an $(\infty, k_j-1)$-polar graph.

So far, we have only shown that $k \le k_1 + k_2 -1$.
To prove the equality, we will show that $G$ is
not a cograph minimal $(\infty, j)$-polar obstruction
for $j \le k_1 + k_2 -2$, which implies that
$k \ge k_1 + k_2 -1$.

It follows from Lemma \ref{lem:CharacterizationOfAitchs}
that $k_i \ge 2$, and by construction we have that
$k_i \le k-1$.   The above observations imply that if
$k \le 2$, then there exist no $(\infty,k)$-polar
obstructions without isolated vertices or components
isomorphic to $P_3$, so we can assume that $k \ge 3$.
Suppose, to reach a contradiction, that $k < k_1 + k_2
-1$, in which case at least one of $k_1$ and $k_2$ is
greater than or equal to three.   Let us assume without
loss of generality that $k_1 \ge 3$.

Since $k_1 \ge 3$ we have from Theorem
\ref{theo:Charact(1,k)PO} and Lemma
\ref{lem:CharacterizationOfAitchs} that $H_1$
contains, as a proper induced subgraph, a cograph $H'_1$
which is not a cluster and such that it is a
$(1, k_1-1)$-polar graph	but it is not
$(1, k_1-2)$-polar.   Observe that the cograph
$G' = H'_1 + H_2$ is not an
$(\infty, k_1 + k_2 - 2)$-polar graph, because since
neither $H'_1$ nor $H_2$ are clusters, every
$(\infty, k_1 + k_2 -2)$-polar partition of $G'$ is a
$(1, k_1 + k_2 -2)$-polar partition, which is impossible
since $H'_1$ is not $(1, k_1-2)$-polar and $H_2$ is not
$(1, k_2-1)$-polar.   Therefore $G$ has a
cograph $(\infty, k_1+k_2 -2)$-polar obstruction as a
proper	induced subgraph, and then $G$ is not a cograph
minimal $(\infty ,j)$-polar obstruction for
$j  < k_1 + k_2 -2$.   As we have mentioned, it proves
that $k = k_1 + k_2 -1$, which is absurd since we are
supposing that $k < k_1 + k_2 -1$.   Thus,
$k = k_1 + k_2 -1$ as we intended.   This finalizes
the proof of the first implication of the proposition.

For the converse implication let us suppose that
$G = H_1 + H_2$ is a cograph without components
isomorphic to $P_3$ such that, for some positive
integers $k_1$ and $k_2$ and any election of
$i,j\in\{1, 2\}$, $i \neq j$, the graphs $H_i$ and
$H_j$ satisfy the enumerated items of this lemma's
statement.

Suppose for a contradiction that $G$ admits an
$(\infty, k)$-polar partition $(A, B)$.   Since $H_1$
and $H_2$ are not clusters and $k = k_1 + k_2 -1$,
$(A, B)$ is a $(1, k_1 + k_2 -1)$-polar partition of
$G$, but this is impossible since for hypothesis $H_i$
is not a $(1, k_i -1)$-polar cograph for any
$i \in \{1,2\}$.   Thus $G$ is not an
$(\infty, k$)-polar graph.

Let $v$ be a vertex of $G$, let us suppose without loss
of generality that $v \in V_{H_1}$.   If $H_1 -v$ admits
a $(1, k_1 -1)$-polar partition $(A_1, B_1)$, then, for
any $(1, k_2)$-polar partition $(A_2, B_2)$ of $H_2$,
$(A_1 \cup A_2, B_1 \cup B_2)$ is a $(1, k)$-polar
partition of $G-v$.   Otherwise, if $H_1 -v$ is not a
$(1, k_1 -1)$-polar graph, by item 3 we have that
$H_1 - v$ is a cluster, and by Lemma
\ref{lem:CharacterizationOfAitchs}
and Remark
\ref{rem:(Infty,k-1)PolarityAndKClusters}
it has exactly $k_1$
components.   In addition, by item 4, $H_2$ is an
$(\infty, k_2 -1)$-polar graph.   Thus, if $(A_1, B_1)$
is a $(0, k_1)$-polar partition of $H_1 -v$ and
$(A_2, B_2)$ is an $(\infty, k_2 -1)$-polar partition of
$H_2$, then $(A_1 \cup A_2, B_1 \cup B_2)$ is an
$(\infty, k)$-polar partition of $G-v$.   Hence, $G$ is
a cograph minimal $(\infty, k)$-polar obstruction.
Evidently, $G$ is a disconnected graph, and it follows
from Lemma \ref{lem:CharacterizationOfAitchs} that $G$
has no isolated vertices.
\end{proof}

Based on Lemma \ref{lem:CharacterizationOfAitchs}, Remark
\ref{rem:(Infty,k-1)PolarityAndKClusters}, and Lemma
\ref{lem:ObsminWhitoutP3CompNorIsolated} it is
straightforward to deduce the following recursive
construction of cograph minimal $(\infty, k)$-polar
obstructions without isolated vertices nor components
isomorphic to $P_3$.

\begin{theorem} 
\label{theo:ObsminWithoutIsolatedNorP3Comp}
Let $k$ be a positive integer, and let $G$ be a graph
without components isomorphic to $P_3$.   Then, $G$ is a
disconnected cograph minimal $(\infty ,k)$-polar obstruction
without isolated vertices if and only for some positive
integers $k_1$ and $k_2$,
and some cographs $H_1$ and $H_2$,
\begin{enumerate}
	\item $G = H_1 + H_2$,
	\item $k = k_1 + k_2 - 1$,
	\item for $i \in \{1,2\}$, $H_i$ is either a cograph
		minimal $(1, k_i -1)$-polar obstruction that is
		neither a cograph minimal $(1, \infty)$-polar
		obstruction nor isomorphic to $k_iK_2$, or
		$k_i \ge 2$ and
		$H \cong (k_i-2)K_2 + (K_1 \oplus 2K_2)$, and
	\item if $H_i \cong (k_i-2)K_2 + (K_1 \oplus 2K_2)$ and
		$G- H_i$ is a cograph minimal $(1, k_i -1)$-polar
		obstruction, then $G - H_i$ has exactly one
		component non-isomorphic to $K_2$.
\end{enumerate}
\end{theorem}

We now turn our attention to cograph minimal
$(\infty, k)$-polar obstructions without isolated
vertices that have some component isomorphic to $P_3$.
We begin with a technical characterization of such
family of graphs, followed by two lemmas that
treat with specific subcases, and finalize with a
recursive construction for these obstructions.
It is worth noticing that, by Lemma
\ref{lem:ObsminWithP3Component}, any cograph minimal
$(\infty, k)$-polar obstruction with a component
isomorphic to $P_3$ has no isolated vertices.

\begin{lemma} 
\label{lem:ObsminWithP3Component}
Let $k$ be a positive integer, and let $G$ be a graph
with at least one component isomorphic to $P_3$.
Then $G$ is a cograph minimal $(\infty, k)$-polar
obstruction if and only if $G \cong P_3 + H$, where
$H$ is a cograph that satisfies the following
statements:
\begin{enumerate}
	\item $H$ is not a $(1, k-1)$-polar graph,
	\item $H$ is not a cluster,
	\item $H$ is an $(\infty, k-1)$-polar graph,
	\item $H$ is either a $(1, k)$-polar graph or an
		$(\infty, k-2)$-polar graph, and
	\item for each vertex $v$ of $H$, the graph $H-v$ is
		either a $(1, k-1)$-polar graph or a $k$-cluster.
\end{enumerate}
\end{lemma}

\begin{proof}
Suppose that $G$ is a cograph minimal
$(\infty, k)$-polar obstruction with a component
isomorphic to $P_3$, and let $H$ be such that
$G \cong P_3 + H$.   Note that $H$ can not be a
$(1, k-1)$-polar graph, because $P_3$ is a $(1,1)$-polar
graph, and then $G$ would be a $(1, k)$-polar graph.

To prove that $H$ is not a cluster we will first prove
by means of a contradiction that $H$ has no isolated
vertices.   If $H$ has at least one isolated vertex,
we have from Lemma \ref{lem:BasicsOf(Infty,k)PO} that
for some positive integers $p$ and $q$, $G$ is
isomorphic to $pK_1 + qK_2 + P_3$, but in such a case
$G$ is a $(1, q+1)$-polar graph, which implies that
$k \le q$.   Furthermore, for each integer
$j\in\{2,\dots,q\}$, $G$ contains the cograph minimal
$(\infty, j)$-polar obstruction $K_1 + (j+1)K_2$ as a
proper induced subgraph, which implies that $k\le 1$.
But it is impossible, since the cograph minimal
$(\infty, k)$-polar obstructions for $k \le 1$ have no
components isomorphic to $P_3$.   Hence, $H$ has no
isolated vertices, and if $H$ is a cluster, then for
some positive integer $q$, $G \cong qK_2 + P_3$.
We have that $q \ge k+1$ because $G$ is not an
$(\infty, k)$-polar graph, but then $G$ contains the
cograph minimal $(\infty, k)$-polar obstruction
$K_1 + (k+1)K_2$ as a proper induced subgraph,
in contradiction with the minimality of $G$.
This contradiction arose from supposing that $H$ is
a cluster, so it is not.

Items 3 to 5 can be easily proved by considering
$(\infty, k)$-polar partitions of $G-v$ when $v$ is
either a vertex of $G-H$, or when $v$ is a vertex of
$H$.   We have used similar arguments before, so the
details of these arguments are omitted. Also, the
proof of the converse implication is very similar
to the proof of the converse of Lemma
\ref{lem:ObsminWhitoutP3CompNorIsolated}, so it will
be also omitted.
\end{proof}

\begin{lemma} 
\label{coro:ObsminWithP3Component}
Let $k$ be a positive integer, and let $G \cong P_3 + H$
be a cograph minimal $(\infty, k)$-polar obstruction.
If $H$ is not a $(1, k)$-polar graph, then $H$ is
a connected non-complete graph.
\end{lemma}

\begin{proof}
Let $k, G$ and $H$ be as in the hypothesis.   By Lemma
\ref{lem:ObsminWithP3Component}, $H$ is not a cluster,
so $H$ has at least one non-complete component.
Moreover, it also follows from Lemma
\ref{lem:ObsminWithP3Component} that $H$ is an
$(\infty, k-1)$-polar graph that is not
$(1, k-1)$-polar, which implies that $H$ cannot have
more than one non-complete component.   Therefore, $H$
has precisely one non-complete component.   In addition,
since $H$ is an induced subgraph of $G$,
it follows from Lemma \ref{lem:BasicsOf(Infty,k)PO}
that every non-complete component of $H$ is isomorphic
to $K_1$ or $K_2$.   Moreover, it also follows from
Lemma \ref{lem:BasicsOf(Infty,k)PO} that $H$ has no
isolated vertices, otherwise $G$ would have at most one
non-complete component, which is not the case.   Hence,
for some non-negative integer $\ell$,
$H \cong \ell K_2 + H'$, where $H'$ is a connected
non-complete graph.

Suppose that $\ell \ge 1$, and let $v \in V_{H-H'}$.
Note that since $H' \le H-v$, we have that $H-v$ is
not a cluster. Hence, by Lemma
\ref{lem:ObsminWithP3Component},
we have that $H-v$ admits a $(1,k-1)$-polar partition
$(A, B)$.   But in such case, $(A, B\cup\{v\})$ is a
$(1, k)$-polar partition of $H$, which is impossible
from our original hypotheses.   The contradiction arose
from supposing that $\ell \ge 1$, so $\ell = 0$ and then
$H = H'$, which proves that $H$ is a connected
non-complete graph.
\end{proof}

The next trivial observation will be helpful in some
of the following results.   It is immediate from the
cotree representation of cographs.

\begin{remark} 
\label{rem:K1OplusH}
Let $H$ be a connected cograph, and let $H'$ be a
disconnected induced subgraph of $H$.   Then
$K_1 \oplus H'$ is also an induced subgraph of $H$.
\end{remark}

\begin{lemma} 
\label{cor:ObsminType(2,0)}
Let $k$ be a positive integer, and let $G$ be a graph
with at least one component isomorphic to $P_3$.
\begin{enumerate}
	\item If $k=2$ then, $G$ is a cograph minimal
		$(\infty, k)$-polar obstruction of type $(2,0)$ if
		and only if $G \cong P_3 + C_4$ or
		$G \cong P_3 + (K_1 \oplus 2K_2)$.
	\item If $k \ge 3$ then, $G$ is a cograph minimal
		$(\infty, k)$-polar obstruction of type $(2,0)$ if
		and only if $G \cong P_3 + H$, where $H$ is
		any connected cograph minimal $(1, k-1)$-polar
		obstruction.
\end{enumerate}
\end{lemma}

\begin{proof}
We prove only the second statement, the case $k = 2$ can
be treated in a very similar way.   Suppose that $H$
is any connected cograph minimal $(1, k-1)$-polar
obstruction.   Since $k \ge 3$, we have from Theorems
\ref{theo:Charact(1,infty)PO} and
\ref{theo:Charact(1,k)PO} that $H$ is isomorphic to one
cograph in the set
\[\{ K_1 \oplus C_4,\ K_2 \oplus 2K_2,\
\overline{2P_3},\ K_1 \oplus (K_2 + P_3),\ K_{k, k},\
\overline{K_2} \oplus (K_2 + (k-1)K_1), \
K_1 \oplus (2K_2 + (k-2)K_1)\}.
\]

It is straightforward to check that in any case $H$
satisfies the items enumerated in Lemma
\ref{lem:ObsminWithP3Component}, which implies that
$P_3 + H$ is a cograph minimal $(\infty, k)$-polar
obstruction of type $(2,0)$.

Conversely, let suppose that $G$ is a cograph minimal
$(\infty, k)$-polar obstruction of type $(2,0)$.
From Lemma \ref{lem:ObsminWithP3Component},
$G \cong P_3 + H$ where $H$ is a connected non-complete
cograph that contains a cograph minimal $(1, k-1)$-polar
obstruction $H'$ as an induced subgraph. As we just
mentioned, if $H'$ is connected, then $P_3 + H'$ is a
cograph minimal $(\infty, k)$-polar obstruction, so
$H = H'$.   Otherwise, if $H'$ is disconnected, it
follows from Remark \ref{rem:K1OplusH} that $H$ contains
$K_1 \oplus H'$ as an induced subgraph.   Nevertheless,
from Theorem \ref{theo:Charact(1,k)PO} and Lemma
\ref{lem:ObsminWithP3Component}, $H' \cong kK_2$, but in
this case $H$ contains properly the connected cograph
minimal $(1, k-1)$-polar obstruction
$K_1 \oplus (2K_2 + (k-2)K_1)$ as an induced subgraph,
which contradicts the minimality of $G$.
\end{proof}

\begin{theorem} 
\label{theo:ObsminWithP3Compo}
Let $k$ be an integer, $k\ge 2$, and let $G$ be a
graph with at least one component isomorphic to $P_3$.
Then, $G$ is a cograph minimal $(\infty, k)$-polar
obstruction if and only if $G \cong P_3 + H$ and one
of the following statements is satisfied:
\begin{enumerate}
	\item $H \cong P_3 + (k-1)K_2$.
	\item $H \cong (k-2)K_2 + (K_1 \oplus 2K_2)$.
	\item for some integer $j \in \{1, \dots, k-1\}$,
		$H \cong (k-j-1)K_2 + H_j$, where $H_j$ is a
		connected cograph minimal $(1, j)$-polar obstruction
		that is not a cograph minimal $(1, \infty)$-polar
		obstruction.
	\item $k \ge 3$, and $H$ is any cograph minimal
		$(1, \infty)$-polar obstruction.
\end{enumerate}
\end{theorem}

\begin{proof}
Suppose that $G \cong P_3 + H$ is a cograph minimal
$(\infty, k)$-polar obstruction.   From Lemmas
\ref{lem:CharacterizationOfAitchs} and
\ref{lem:ObsminWithP3Component}, and Remark
\ref{rem:(Infty,k-1)PolarityAndKClusters}, we have that
if $H$ is a $(1, k)$-polar graph, then $H$ satisfies
one of the following statements:
\begin{enumerate}[label=\alph*.]
	\item $H \cong P_3 + (k-1)K_2$.
	\item $k \ge 2$, and
		$H \cong (k-2)K_2 + (K_1 \oplus 2K_2)$.
	\item $H$ is a cograph minimal $(1, k-1)$-polar
		obstruction, that is neither a cograph minimal
		$(1, \infty)$-polar obstruction nor isomorphic
		to $kK_2$, and such that exactly one of its
		components is not isomorphic to $K_2$.
\end{enumerate}

Furthermore, from Theorem \ref{theo:Charact(1,k)PO} we
have that the graphs described in item c are precisely
the graphs $H$ such that, for some integer
$j \in \{1, \dots, k-1\}$, $H \cong (k-j-1)K_2 + H_j$,
where $H_j$ is a connected cograph minimal
$(1, j)$-polar obstruction that is not a cograph minimal
$(1, \infty)$-polar obstruction.

Suppose then that $H$ is not a $(1,k)$-polar graph.
It follows from Lemma \ref{coro:ObsminWithP3Component}
that $H$ is a connected non-complete graph, and then
$G$ is a cograph minimal $(\infty, k)$-polar obstruction
of type $(2, 0)$, and then it follows from Lemma
\ref{cor:ObsminType(2,0)} that $H$ satisfies item 3 or
item 4 of the theorem statement.

The converse follows easily from Lemmas
\ref{lem:CharacterizationOfAitchs},
\ref{lem:ObsminWithP3Component},
\ref{cor:ObsminType(2,0)},
and Remark
\ref{rem:(Infty,k-1)PolarityAndKClusters}.
\end{proof}

\section{Obstructions of type $(c, c-1)$}
\label{sec:type(c,c-1)}

So far, we have obtained a recursive characterization of
cograph minimal $(\infty, k)$-polar obstructions without
isolated vertices.   Next, we focus on the special case of
cograph minimal $(\infty, k)$-polar obstructions such that
all its components, except one, are trivial.   We begin
giving a technical characterization of these obstructions.
The proof of this result is omitted since it is very
similar to that of Lemma \ref{lem:ObsminWithP3Component}.

\begin{lemma} 
\label{lem:Obsmin(k-j+1, k-j)}
Let $j$ and $k$ be integers such that $0 \le j+1 \le k$.
Then, $G$ is a cograph minimal $(\infty, k)$-polar
obstruction of type $(k-j+1, k-j)$ if and only if
$G \cong (k-j)K_1 + H$, where $H$ is a connected
non-complete cograph that satisfies the following
statements:
\begin{enumerate}
	\item $H$ is not a $(1, k)$-polar graph,
	\item $H$ is not an $(\infty, j)$-polar graph,
	\item $H$ is an $(\infty, j+1)$-polar graph,
	\item for each vertex $v$ of $H$, $H-v$ is either
		$(1,k)$-polar or $(\infty, j)$-polar.
\end{enumerate}
\end{lemma}

The following result provides a pleasant recursive
characterization of cograph minimal $(\infty, k)$-polar
obstructions of type $(2,1)$.

\begin{theorem} 
\label{theo:MinObsType(2,1)}
Let $k$ be an integer, $k \ge 2$, and let $G$ be a graph
with precisely two connected components, one of them
trivial. Then $G$ is a cograph minimal $(\infty,
k)$-polar obstruction if and only if $G \cong K_1 +
(K_1 \oplus H')$, where $H'$ is a disconnected cograph
minimal $(\infty, k-1)$-polar obstruction that is
$(1,k)$-polar.
\end{theorem}

\begin{proof}
Let $H = K_1 \oplus H'$, where $H'$ is a disconnected
cograph minimal $(\infty, k-1)$-polar obstruction that
is $(1,k)$-polar, and let $G = K_1 + H$.   Observe that
by the election of $H'$, $H$	is an $(\infty, k)$-polar
graph that is not $(1,k)$-polar, and for each vertex $v$
of $H$, the graph $H-v$ is either $(1,k)$- or
$(\infty, k-1)$-polar.

Let us suppose for a contradiction that $G$ admits an
$(\infty, k)$-polar partition $(A,B)$.   Since $H$
is not $(1,k)$-polar, $G[A]$ must be a non-trivial
connected graph.   Thus, since	$H$ is not complete,
$A \subseteq V_{H}$ and $H$ is an $(\infty, k-1)$-polar
graph, which is impossible.   Therefore $G$ is not an
$(\infty, k)$-polar graph.

Let $v$ be a vertex of $G$.   If $v$ is the only
isolated vertex of $G$, then $G-v = H$, and $H$ is an
$(\infty, k)$-polar graph, so $G-v$ is.   Otherwise
$v \in V_{H}$	and, as we have noted above, $H-v$ is
is either	$(1,k)$- or $(\infty, k-1)$-polar, so
$G-v$ is an $(\infty, k)$-polar graph.  Hence,
$G \cong K_1 + (K_1 \oplus H')$ is a cograph minimal
$(\infty, k)$-polar obstruction whenever $H'$ is a
cograph minimal $(\infty, k-1)$-polar obstruction that
is $(1,k)$-polar.

Conversely, suppose that $G$ is a cograph minimal
$(\infty, k)$-polar obstruction with precisely two
connected components, one of them trivial.
By Lemma \ref{lem:Obsmin(k-j+1, k-j)}, $G \cong K_1 + H$
for some connected cograph $H$ that is not
$(\infty ,k-1)$-polar, and such that for every vertex
$v$ of $H$, $H-v$ is either $(1,k)$- or
$(\infty, k-1)$-polar.   Note that $H$ contains a
cograph minimal $(\infty, k-1)$-polar obstruction $H'$
as an induced subgraph, but from Theorem
\ref{theo:ConnectedObsmin}, $H'$ can not be connected,
or $H'$ would be a proper induced subgraph of $G$ that
is not an $(\infty, k)$-polar graph, an absurd.   Thus,
$H'$ must be disconnected, and from Remark
\ref{rem:K1OplusH}, $K_1 \oplus H' \le H$.   But in such
a case $G$ contains the cograph minimal
$(\infty, k)$-polar obstruction $K_1 + (K_1 \oplus H')$
as an induced subgraph, so
$G \cong K_1 + (K_1 \oplus H')$.
\end{proof}

Notice that, by Lemma \ref{lem:Obsmin(k-j+1, k-j)},
a graph $G$ is a cograph minimal $(\infty, k)$-polar
obstruction of type $(k+2, k+1)$ if and only if
$G \cong (k+1)K_1 + H$, where $H$ is a connected
non-complete cograph minimal $(1, k)$-polar
obstruction that is a complete multipartite graph.
Moreover, from Theorem \ref{theo:Charact(1,k)PO},
for any integer $k$, $k \ge 2$, the only cograph
minimal $(1,k)$-polar obstructions that are
complete multipartite graphs are $K_{k+1, k+1}$ and
$K_1 \oplus C_4$.   Thus, the following result follows
immediately from Lemma \ref{lem:Obsmin(k-j+1, k-j)}.

\begin{theorem} 
\label{cor:ObsminType(k+2,k+1)}
Let $k$ be an integer, $k \ge 2$.   Thus, $G$ is a cograph
minimal $(\infty, k)$-polar obstruction of type $(k+2,k+1)$
if and only of $G \cong (k+1)K_1 + H$, where $H$ is
isomorphic to $K_{k+1,k+1}$ or to $K_1 \oplus C_4$.
\end{theorem}

Unfortunately, obtaining explicit lists of disconnected
cograph minimal $(\infty ,k)$-polar obstructions with
precisely one non-trivial component is a very difficult
task. As we have shown above, a simple recursive
construction of $(2,1)$- and $(k+2, k+1)$-type
obstructions is possible, but as we show below, it is
not enough for covering the general case.

The following three lemmas are auxiliary
results that will be the cornerstone for
obtaining explicit lists of cograph minimal
$(\infty, k)$-polar obstructions of types
$(k+1,k)$ and $(k, k-1)$.   The three of
them are based on the same proof technique:
we consider all the distinct ways in which a
connected cograph can be generated from
another connected cograph whose cotree has
specific characteristics.   For the sake of
length, we only sketch the proof of the first
proposition, the other two are very similar.

\begin{lemma} 
\label{lem:ExtendCompleteMultipartite1}
Let $k$ be an integer, $k \ge 2$, and let $H$
be a complete multipartite graph with at least
two parts and such that each part has at least
$k+1$ vertices.    If $H'$ is a connected cograph
of order $V_{H} +1$ that contains $H$ as an induced
subgraph, then exactly one of the following sentences
is satisfied:
\begin{enumerate}
	\item $H'$ is a complete multipartite graph with at
		least two parts and such that each part has at least
		$k+1$ vertices.

	\item $H'$ contains $\overline{K_2} \oplus (K_2 + kK_1)$
		or $K_1 \oplus C_4$ as an induced subgraph.
\end{enumerate}
\end{lemma}

\begin{proof}
Let $k,H$ and $H'$ be as in the hypotheses.

Note that the cotree of $H$ is a rooted tree
$(T,r)$ of height two, such that $r$ is labeled 1,
$r$ has at least two children, and each of them is
the parent of at least $k+1$ leaves.   Then, by the
properties of cotrees, the cotree of $H'$ is
a rooted tree $(T',r)$ with exactly one more leaf
than $T$, that contains $T$ as an induced tree.

It can be verified that a tree $T'$ as
described above is necessarily the result of one
of the following modifications on $T$:
(a) adding a new leaf $x$ as a child of $r$,
(b) adding a new leaf $x$ as a child of a child of
    $r$,
(c) for a child $c$ of $r$, deleting a child $\ell$
    of $c$, adding a child $c'$ to $c$, and adding
    to $c'$ the leave $\ell$ and a new leaf $x$,
(d) for a child $c$ of $r$ with $t$ children, and
    $s\in\{2,\dots,t-1\}$, deleting the children
    $\ell_1, \dots, \ell_s$ of $c$, adding a child
    $c'$ to $c$, adding a new leaf $x$ as a child
    of $c'$, adding a child $c''$ to $c'$, and
    adding the leaves $\ell_1, \dots, \ell_s$ as
    children of $c''$, or
(e) supposing $r$ has $t$ children, for an integer
    $s\in\{2,\ldots,t-1\}$, deleting the children
    $c_1,\ldots,c_s$ of $r$ (each with its own
    children), adding a child $c'$ to $r$, adding
    a new leaf $x$ to $c'$, adding a child $c''$
    to $c$, and adding the vertices $c_1,\ldots,c_s$
    (with its children) as children of $c''$.

It is straightforward to corroborate that such
modifications on $T$ correspond to the following
modifications on $H$:

\begin{enumerate}[label=\alph*.]
	\item Add a universal vertex to $H$.
	\item Add a false twin to a vertex of $H$.
	\item Add a true twin to a vertex of $H$.
	\item Add a vertex $v$ to $H$ in such a way that $v$
		is completely adjacent to every part of $H$,
		except for a part $P$, and $v$ is adjacent to at
		least two vertices in $P$ but it is not adjacent
		to every vertex of $P$.
	\item Add a vertex $v$ to $H$ in such a way that $v$
		is completely non-adjacent to at least two parts
		of $H$, and it is completely adjacent to at least
		one part of $H$.
\end{enumerate}

Then, if $H'$ corresponds to the operation described in c,
then $H$ has $\overline{K_2} \oplus (K_2 + kK_1)$ as a
proper induced subgraph, while if $H$ corresponds to an
operation described in items a, d, or e, then $H'$
contains $K_1 \oplus C_4$ as an induced subgraph, and
if $H'$ is obtained from the operation described in
item b, then $H$ is a complete multipartite graph with
at least two parts and such that each part contains at
least $k+1$ vertices.
\end{proof}

\begin{lemma} 
\label{lem:ExtendCompleteMultipartite2}
Let $H$ be a complete multipartite graph with at
least three parts and such that at least two of
them have more than one vertex.   If $H'$ is a
connected cograph obtained by adding a new vertex
to $H$, then exactly one of the following
conditions is satisfied:
\begin{enumerate}
	\item $H'$ is a complete multipartite graph.
	\item $H'$ contains, as an induced subgraph,
		at least one of the following cographs: $K_1
		\oplus (K_1 + C_4), \overline{K_2}\oplus (K_1
		+ P_3)$, or $K_1 \oplus (\overline{P_3 + K_2})$.
\end{enumerate}
\end{lemma}

\begin{lemma} 
\label{lem:ExtendCograph}
Let $k$ be an integer, $k \ge 3$, and let $H$ be a
connected $(1,k)$-polar cograph that contains
$K_1 \oplus (2K_2 + K_1)$ as an induced subgraph.
If $H'$ is a connected cograph obtained by adding a new
vertex to $H$, then some of the following sentences is
satisfied:
\begin{enumerate}
	\item $H'$ is a $(1,k)$-polar cograph.

	\item $H'$ contains some of the following cographs as an
		induced subgraph: $K_1 \oplus (2K_2+ (k-1)K_1),
		K_2 \oplus (2K_2 + K_1),
		K_1 \oplus (P_3 + \overline{P_3}),
		K_1 \oplus (K_2 + \overline{K_1 + P_3})$, or
		$K_1 \oplus (K_1 + (K_1 \oplus 2K_2))$.
\end{enumerate}
\end{lemma}

Now we are ready to give explicit lists of
cograph minimal $(\infty, k)$-polar obstructions
of types $(k+1, k)$ and $(k, k-1)$.   As we have
mentioned above, these lists are directly based
on the previous lemmas.

\begin{corollary} 
\label{cor:ObsminType(k+1,k)}
Let $k$ be an integer, $k \ge 2$.   Then, $G$ is a cograph
minimal $(\infty, k)$-polar obstruction of type $(k+1, k)$
if and only if $G \cong kK_1 + H$, where $H$ is isomorphic
to some cograph of the set:
\[
\{\overline{2P_3},
(P_3 + K_2) \oplus K_1,
2K_2 \oplus K_2,
K_1 \oplus (C_4 + K_1),
K_1 \oplus \overline{P_3 + K_2},
\overline{K_2} \oplus (P_3 + K_1),
\overline{K_2} \oplus (K_2 + kK_1)\}.
\]
\end{corollary}

\begin{proof}
By Lemma \ref{lem:Obsmin(k-j+1, k-j)}, it is routine
to verify that if $H$ is isomorphic to some of the
listed graphs, then $kK_1 + H$ is a cograph
minimal $(\infty, k)$-polar obstruction.   For the
converse, let us consider $G$, a cograph
minimal $(\infty, k)$-polar obstruction of type
$(k+1, k)$. By Lemma \ref{lem:Obsmin(k-j+1, k-j)}
we have that $G \cong kK_1 + H$, where $H$ is a
connected non-complete cograph that contains a
cograph minimal $(1,k)$-polar obstruction $H'$ as
an induced subgraph, and such that for each vertex
$v \in V_{H}$, $H-v$ is either a $(1,k)$-polar graph
or a complete	multipartite graph.   In addition,
by Theorem \ref{theo:Charact(1,k)PO} we know that
every disconnected cograph minimal $(1,k)$-polar
obstruction is not a complete multipartite graph,
which implies from Remark \ref{rem:K1OplusH} that
$H'$ can not be disconnected.   Then, since $k \ge
2$, we have that $H'$ is either isomorphic to some
graph of the set $\{K_{k+1,k+1}, K_1 \oplus (2K_2
+ (k-1)K_1), \overline{K_2}\oplus(K_2 + kK_1)\}$,
or it is isomorphic to some $(1, \infty)$-polar
obstruction, that is, to some graph of the set
$\{K_1\oplus C_4, K_2\oplus 2K_2, \overline{2P_3},
K_1 \oplus (P_3 + K_2)\}$.

We observed at the beginning of this proof that if
$H'$ is isomorphic to $K_2\oplus 2K_2,
\overline{2P_3}, K_1 \oplus (P_3 + K_2)$ or
$\overline{K_2}\oplus(K_2 + kK_1)$, then $kK_1 + H'$ is
a cograph minimal $(\infty, k)$-polar obstruction, so
in this cases $H = H'$.
Furthermore, since $k \ge 2$, Lemma
\ref{lem:Obsmin(k-j+1, k-j)} implies that
$G' \cong (k-1)K_1 + K_1 \oplus (2K_2 +(k-1)K_1)$ is a
cograph minimal $(\infty, k)$-polar obstruction.   In
consequence, $H' \not\cong K_1 \oplus (2K_2 +(k-1)K_1)$,
or	$G$ would contain $G'$ as a proper induced subgraph,
a contradiction.

Thus, we have only two remaining cases,
$H' \cong K_1 \oplus C_4$, or $H'\cong K_{k+1, k+1}$.
Note that in both cases, $H'$ is a complete multipartite
graph, and by Lemma \ref{lem:Obsmin(k-j+1, k-j)}
$H$ is not a complete multipartite graph,
so $H'$ must be a proper induced subgraph of $H$.
Furthermore, by Theorem \ref{cor:ObsminType(k+2,k+1)},
in both cases, $G' \cong (k+1)K_1 + H'$ is a cograph
minimal $(\infty, k)$-polar obstruction, which implies
that for each vertex $v \in V_{H-H'}$, $v$ is adjacent
to some vertex of $H'$.

Suppose that $H' \cong K_1 \oplus C_4$.   As we have
mentioned before, it is straightforward to show that
$kK_1 + K_1 \oplus (C_4 + K_1)$,
$kK_1 + K_1 \oplus \overline{P_3 + K_2}$ and
$kK_1 + \overline{K_2} \oplus (P_3 + K_1)$ are all
cograph minimal $(\infty, k)$-polar obstructions, so, if
$H$ contains as an induced subgraph a graph $H^*$ that
is isomorphic to either $K_1 \oplus (C_4 + K_1),
K_1 \oplus \overline{P_3 + K_2}$, or
$\overline{K_2} \oplus (P_3 + K_1)$, then $kK_1 + H^*$
is a cograph minimal $(\infty, k)$-polar obstruction
contained as an induced subgraph in $G$, and then
$G \cong kK_1 + H^*$ and $H \cong H^*$.   Moreover,
from Lemma \ref{lem:ExtendCompleteMultipartite2}, $H$
must contain as an induced subgraph a graph $H^*$ as
described before, or $H$ would be a complete
multipartite graph, a contradiction.

For the last case, suppose that $H' \cong K_{k+1, k+1}$.
Then, from Lemma \ref{lem:ExtendCompleteMultipartite1}
and since $H$ is not a complete multipartite graph, $H$
either contains $K_1 \oplus C_4$ as an induced subgraph,
or it contains $\overline{K_2}\oplus(K_2 + kK_1)$ as a
proper induced subgraph.   Since we have already treated
both cases before, we conclude that the only cograph
minimal $(\infty, k)$-polar obstructions of type
$(k+1 ,k)$ are the listed one in the statement of
the corollary.
\end{proof}

\begin{corollary} 
\label{cor:ObsminType(k,k-1)}
Let $k$ be an integer, $k \ge 3$.   The graph
$G$ is a cograph minimal $(\infty, k)$-polar
obstruction of type $(k, k-1)$ if and only if
$G \cong (k-1)K_1 + H$, where $H$ is isomorphic
to some cograph of the set

$\{K_1 \oplus (C_4 + 2K_1),
K_1 \oplus 2P_3,
K_1 \oplus (K_1 + \overline{P_3 + K_2}),
K_1 \oplus (K_2 + \overline{P_3 + K_1}),
K_2 \oplus (K_1 + 2K_2),$
\hspace*{\fill}$K_1 \oplus (K_1 + K_2 + P_3),
K_1 \oplus (K_1 + (K_1 \oplus 2K_2)),
K_1 \oplus ((k-1)K_1 + 2K_2)
\}.$
\end{corollary}

\begin{proof}
Based on Lemma \ref{lem:Obsmin(k-j+1, k-j)}, it is
routine to verify that if $H$ is isomorphic to some
of the listed graphs, then $(k-1)K_1 + H$ is a
cograph minimal $(\infty, k)$-polar obstruction.

Conversely, let $G$ be a cograph minimal
$(\infty, k)$-polar obstruction of type $(k, k-1)$.
By Lemma \ref{lem:Obsmin(k-j+1, k-j)} we have that
$G \cong (k-1)K_1 + H$, where $H$ is a connected
cograph that contains a cograph minimal
$(\infty, 1)$-polar obstruction $H'$ as an induced
subgraph.   Thus, from Theorem
\ref{theo:Charact(1,infty)PO} and Remark
\ref{rem:K1OplusH}, there exists a cograph $H'$
isomorphic to some graph in the set
\[
	\{K_1 \oplus (K_1 + 2K_2), \  K_1 \oplus (2K_1 + C_4),
	\ K_1 \oplus 2P_3,
	\ K_1 \oplus (K_1 + \overline{K_2 + P_3})\}
\]
contained as an induced subgraph of $H$. As we have
observed at the start of this proof, if $H'$ is isomorphic
to either $K_1 \oplus (2K_1 + C_4), K_1 \oplus 2P_3$
or $K_1 \oplus (K_1 + \overline{K_2 + P_3})$,
then $(k-1)K_1 + H'$ is a cograph minimal
$(\infty ,k)$-polar obstruction, and then $H = H'$.
Suppose then that $H' \cong K_1 \oplus (K_1 + 2K_2)$.
It follows from Lemma \ref{lem:Obsmin(k-j+1, k-j)} that
$H$ is not a $(1, k)$-polar graph, which implies from
Lemma \ref{lem:ExtendCograph} that $H$ contains a graph
$H^*$ in the set

$\{ K_1 \oplus (2K_2+ (k-1)K_1),
K_2 \oplus (2K_2 + K_1),
K_1 \oplus (P_3 + \overline{P_3}),
K_1 \oplus (K_2 + \overline{K_1 + P_3}),$

\hspace*{\fill}$K_1 \oplus (K_1 + (K_1 \oplus 2K_2))\}$

as an induced subgraph.   Since we have proved that in
every such case $(k-1)K_1 + H^*$ is a cograph minimal
$(\infty, k)$-polar obstruction, we have that $H = H^*$,
which finishes the proof.
\end{proof}

\section{The remaining types}
\label{sec:remainingTypes}

In contrast with the obstructions with precisely
one non-trivial component, we show in the following
proposition that cograph minimal $(\infty, k)$-polar
obstructions that have at least one trivial component,
and at least one complete non-trivial component, can
be nicely obtained from the cograph minimal $(\infty,
k-1)$-polar obstructions with at least one isolated
vertex.

\begin{theorem} 
\label{theo:(infty,k)obsminFrom(infty,k-1)obsmin}
Let $j, k$ and $p$ be non-negative integers such that
$1 \le p \le k-j$. The graph $G$ is a cograph minimal
$(\infty, k)$-polar obstruction of type $(k-j+2, p)$
if and only if $G \cong K_2 + G'$ where $G'$ is a cograph
minimal $(\infty, k-1)$-polar obstruction of type
$(k-j+1, p)$ which is a $(1,k)$-polar graph.
\end{theorem}

\begin{proof}
Suppose that $G'$ is a cograph minimal
$(\infty, k-1)$-polar obstruction that is a
$(1, k)$-polar graph, and let $G = K_2 + G'$.
Note that since $G'$ is not an $(\infty, k-1)$-polar
graph, $G$ is not an $(\infty, k)$-polar graph.
Moreover, for $v\in V_{G-G'}$, since $G'$ is a
$(1,k)$-polar graph, $G-v$ is also a $(1,k)$-polar
graph, while for $w\in G'$, since $G'-w$ is an
$(\infty, k-1)$-polar graph, we have that $G-w$ is an
$(\infty, k)$-polar graph.   Thus, $G$ is a cograph
minimal $(\infty, k)$-polar obstruction.

Conversely, let $G$ be a cograph minimal
$(\infty, k)$-polar obstruction of type $(k-j+2, p)$, so
$G \cong pK_1 + (k-j-p+1)K_2 + H$, where $H$ is a
connected non-trivial graph.   Thus, for
$G' = pK_1 + (k-j-p)K_2 + H$, we have that
$G \cong K_2 + G'$.   Observe that since $G$ is not an
$(\infty,k)$-polar graph, $G'$ is not an
$(\infty, k-1)$-polar graph.

Let $v \in V_{G-G'}$, and let $w$ be the only neighbour
of $v$ in $G$. Let $(A,B)$ be an $(\infty,k)$-polar
partition of $G-v$. Note that $w$ must belong to $A$,
or else $G$ would be an $(\infty, k)$-polar graph. Thus
$(A, B)$ is a $(1,k)$-polar partition of $G-v$, and
then $G'$ is a $(1, k)$-polar graph.   Hence, since $G'$
is not $(\infty, k-1)$-polar but it is $(1,k)$-polar,
$G'$ contains a cograph minimal $(\infty, k-1)$-polar
obstruction $G^*$ that is $(1, k)$-polar as an induced
subgraph, but we have shown at the beginning of the
proof that in such a case $K_2 + G^*$ is a cograph
minimal $(\infty, k)$-polar obstruction, so we have
that $G' = G^*$, which finishes the proof.
\end{proof}

A somewhat surprising consequence of the
previous results is that for $c\in\{k, k+1,
k+2\}$ and $i\in\{1, \dots, c-2\}$, there
exists exactly one cograph minimal $(\infty,
k)$-polar obstruction of type $(c, i)$.
We conjecture that the same is true for
any integer $k$ and every integer $c$ such
that $3 \le c \le k+2$.   In the following
proposition we specify the known cases.

\begin{corollary} 
\label{cor:ObsminGeq*k-1Comp}
Let $p$ and $k$ be non-negative integers.
\begin{enumerate}
	\item If $1 \le p \le k+1$, then the graph
		$pK_1 + (k-p+1)K_2 + K_{p,p}$ is a cograph
		minimal $(\infty, k)$-polar obstruction.
		Moreover, for $p \le k$, up to isomorphism,
		this is the only cograph minimal $(\infty,
		k)$-polar obstruction of type $(k+2, p)$.

	\item If $1 \le p \le k$, the graph
		$pK_1 + (k-p)K_2 + (\overline{K_2}
		\oplus (K_2 + pK_1))$ is a cograph minimal
		$(\infty, k)$-polar obstruction.   Moreover,
		for $p \le k-1$, up to isomorphism, this
		is the only cograph minimal $(\infty,
		k)$-polar obstruction of type $(k+1, p)$.

	\item If $1 \le p \le k-1$, the graph
		$pK_1 + (k-p-1)K_2 + (K_1 \oplus (2K_2 +
		pK_1)$ is a cograph minimal $(\infty,
		k)$-polar obstruction.   Moreover, for $p
		\le k-2$, up to isomorphism, this is the
		only cograph minimal $(\infty, k)$-polar
		obstruction of type $(k, p)$.
\end{enumerate}
\end{corollary}

\begin{proof}
Let $k$ and $p$ non-negative integers such that
$1 \le p \le k+1$.   It is routine to show that
$pK_1 + (k-p+1)K_2 + K_{p,p}$ is a cograph minimal
$(\infty ,k)$-polar obstruction that admits a
$(1, k+1)$-polar partition.

Suppose that $p \le k$.   We proceed by mathematical
induction on $k$ to show that the only cograph minimal
$(\infty ,k)$-polar obstruction of type $(k+2, p)$ is
isomorphic to $pK_1 + (k-p+1)K_2 + K_{p,p}$.

The base case, $k = 1$, follows from Theorem
\ref{theo:Charact(1,infty)PO}.   For the inductive step,
suppose that $k \ge 2$, and let $G$ be a cograph minimal
$(\infty, k)$-polar obstruction of type $(k+2, p)$.
Theorem \ref{theo:(infty,k)obsminFrom(infty,k-1)obsmin}
implies that $G \cong K_2 + G'$, where $G'$ is a cograph
minimal $(\infty, k-1)$-polar obstruction of type
$(k+1,p)$ that admits a $(1, k)$-polar partition.
If $p \le k-1$, the induction hypothesis implies that
$G' \cong pK_1 + (k-p)K_2 + K_{p,p}$, where the result
is immediate. In a similar way, if $p = k$, Theorem
\ref{cor:ObsminType(k+2,k+1)} implies that
$G' \cong pK_1 + (k-p)K_2 + K_{p,p}$, which ends the
proof of the first item. The proof of items 2 and 3
are analogous to the proof of item 1, but using
Corollaries \ref{cor:ObsminType(k+1,k)} and
\ref{cor:ObsminType(k,k-1)} instead of Theorem
\ref{cor:ObsminType(k+2,k+1)}.
\end{proof}

\section{Main results}\label{sec:MainResults}

Now, we are ready to state our main results.

\begin{theorem} 
\label{thm:(infty,2)obsmin}
Let $G$ be a cograph minimal $(\infty, 2)$-polar
obstruction. Then
	\begin{enumerate}
	\item $G$ is connected if and only if
		$\overline{G} \cong P_3 + H$, where $H$ is
		isomorphic to $K_1 \oplus C_4$,
		$\overline{2P_3}$, $K_2 \oplus 2K_2$,
		or $K_1 \oplus (K_2 + P_3)$,
	\item $G$ is disconnected and has no isolated vertices
		if and only if $G$ is isomorphic to $P_3 + C_4$,
		$P_3 + (K_1 \oplus 2K_2)$, or $2P_3 + K_2$,
	\item $G$ has exactly $4$ connected components if and
		only if $G$ is isomorphic to
		$3K_1 + (K_1 \oplus C_4)$, or
		$pK_1 + (3-p)K_2 + K_{p,p}$ for some integer $p$
		with $p \in \{1,2,3\}$,
	\item $G$ has exactly $3$ connected components and at
		least one isolated vertex if and only if $G$ is
		isomorphic to $2K_1 + \overline{2P_3}$,
		$2K_1 + (K_1 \oplus (P_3 + K_2))$,
		$2K_1 + (K_2 \oplus 2K_2)$,
		$2K_1 + (K_1 \oplus (C_4 + K_1))$,
		$2K_1 + (K_1 \oplus \overline{P_3 + K_2})$,
		$2K_1 + (\overline{K_2} \oplus (P_3 + K_1))$, or
		$pK_1 + (2-p)K_2 + (2K_1 \oplus (K_2 + pK_1))$ for
		some integer $p$ with $p \in \{1,2\}$,
	\item $G$ has exactly $2$ connected components and
		one isolated vertex if and only if $G$ is
		isomorphic to $K_1 + (K_1 \oplus (K_1 + 2K_2))$,
		$K_1 + (K_1 \oplus (2K_1 + C_4))$,
		$K_1 + (K_1 \oplus 2P_3)$, or
		$K_1 + (K_1 \oplus (K_1 + \overline{K_2 + P_3})).$
	\end{enumerate}
In conclusion, a graph $G$ is a cograph minimal
$(\infty, 2)$-polar obstruction if and only if it is
isomorphic to some of the 23 cographs listed before.
\end{theorem}

\begin{proof}
Item 1 follows from Theorems
\ref{theo:CharactPolarCographs} and
\ref{theo:ConnectedObsmin}, item 2 follows
from Theorems \ref{theo:ObsminWithoutIsolatedNorP3Comp}
and \ref{theo:ObsminWithP3Compo}, items 3 and 4 follow
from Theorem \ref{cor:ObsminType(k+2,k+1)}, Corollaries
\ref{cor:ObsminType(k+1,k)} and
\ref{cor:ObsminGeq*k-1Comp}, and item 5 follows from
Theorem \ref{theo:MinObsType(2,1)}.

By Lemma \ref{lem:BasicsOf(Infty,k)PO} we have that
every cograph minimal $(\infty, 2)$-polar obstruction
has at most $4$ connected components, so the listed
graphs are all the cograph minimal $(\infty, 2)$-polar
obstructions.
\end{proof}

\begin{theorem} 
Let $G$ be a cograph minimal $(\infty, 3)$-polar
obstruction. Then
	\begin{enumerate}
	\item $G$ is connected if and only if
		$\overline{G} \cong P_3 + H$, where $H$ is
		isomorphic to $K_1 \oplus C_4$, $\overline{2P_3}$,
		$K_2 \oplus 2K_2$, or $K_1 \oplus (K_2 + P_3)$,

	\item $G$ is disconnected and has neither isolated
		vertices nor components isomorphic to $P_3$ if and
		only if $G$ is isomorphic to $2C_4$,
		$2(K_1 \oplus 2K_2)$, or $C_4 + (K_1 \oplus 2K_2)$,

	\item $G$ is disconnected and has at least one component
		isomorphic to $P_3$ if and only if $G \cong P_3 +H$,
		where $H$ is isomorphic to $P_3 + 2K_2$,
		$K_2 + (K_1 \oplus 2K_2)$,
		$K_2 + C_4$, $K_{3,3}$,
		$2K_1 \oplus (K_2 + 2K_1)$,
		$K_1 \oplus (K_1 + 2K_2)$, $K_1 \oplus C_4$,
		$\overline{2P_3}$, $K_2 \oplus 2K_2$, or
		$K_1 \oplus (K_2 + P_3)$,

	\item $G$ has exactly $5$ connected components if and
		only if $G$ is isomorphic to
		$4K_1 + (K_1 \oplus C_4)$, or
		$pK_1 + (4-p)K_2 + K_{p,p}$ for some integer $p$
		with $p \in \{1,2,3,4\}$,

	\item $G$ has exactly $4$ connected components and at
		least one isolated vertex if and only if either
		$G \cong K_1 + H$, where $H$ is
		isomorphic to $\overline{2P_3}$,
		$K_1 \oplus (P_3 + K_2)$,
		$K_2 \oplus 2K_2$,
		$K_1 \oplus (C_4 + K_1)$,
		$K_1 \oplus \overline{P_3 + K_2}$, or
		$\overline{K_2} \oplus (P_3 + K_1)$, or $G$ is
		isomorphic to
		$pK_1 + (3-p)K_2 + (2K_1 \oplus (K_2 + pK_1))$ for
		some integer $p$ with $p \in \{1,2,3\}$,

	\item $G$ has exactly $3$ connected components and at
		least one isolated vertex if and only if either
		$G \cong 2K_1 + H$, where $H$ is isomorphic to
		$K_1 \oplus (C_4 + 2K_1)$,
		$K_1 \oplus 2P_3$,
		$K_1 \oplus (K_1 + \overline{P_3 + K_2})$,
		$K_1 \oplus (K_2 + \overline{P_3 + K_1})$,
		$K_2 \oplus (K_1 + 2K_2)$,
		$K_1 \oplus (K_1 + K_2 + P_3)$, or
		$K_1 \oplus (K_1 + (K_1 \oplus 2K_2))$, or
		$G$ is isomorphic to
		$pK_1 + (2-p)K_2 + (K_1 \oplus (2K_2 + pK_1))$
		 for some integer $p$ with $p \in \{1,2\}$,

	\item $G$ has exactly $2$ connected components and
		one isolated vertex if and only if
		$G \cong K_1 + (K_1 \oplus H)$, where $H$ is
		isomorphic to $P_3 + C_4$,
		$P_3 + (K_1 \oplus 2K_2)$, $2P_3 + K_2$,
		$K_1 + 3K_2$, $2K_1 + K_2 + C_4$,
		$3K_1 + K_{3,3}$,
		$K_1 + K_2 + (2K_1 \oplus (K_2 + K_1))$,
		$2K_1 + (2K_1 \oplus (K_2 + 2K_1))$, or
		$K_1 + (K_1 \oplus (K_1 + 2K_2))$.
	\end{enumerate}

In conclusion, a graph $G$ is a cograph minimal
$(\infty, 3)$-polar obstruction if and only if it is
isomorphic to one of the 49 cographs listed before.
\end{theorem}

\begin{proof}
Item 1 follows from Theorems
\ref{theo:CharactPolarCographs} and
\ref{theo:ConnectedObsmin}, item 2 follows
from Theorem \ref{theo:ObsminWithoutIsolatedNorP3Comp},
item 3 follows from Theorem \ref{theo:ObsminWithP3Compo},
items 4, 5 and 6 follow from Theorem
\ref{cor:ObsminType(k+2,k+1)}, and Corollaries
\ref{cor:ObsminType(k+1,k)},
\ref{cor:ObsminType(k,k-1)} and
\ref{cor:ObsminGeq*k-1Comp}, and item 7 follows from
Theorems \ref{theo:MinObsType(2,1)} and
\ref{thm:(infty,2)obsmin}.

From Lemma \ref{lem:BasicsOf(Infty,k)PO} we have that
every cograph minimal $(\infty, 3)$-polar obstruction
has at most $5$ connected components, so the listed
graphs are all the cograph minimal $(\infty, 3)$-polar
obstructions.
\end{proof}


\section{Conclusions}\label{sec:Conclusions}


Exact lists of cograph minimal $(\infty, k)$-polar
obstructions are known when $k \le 1$.   In this
work we present many results focused on the recursive
construction of cograph minimal $(\infty, k)$-polar
obstructions for an arbitrary integer $k$.   After
introducing a simple classification of the obstructions
based on their number of connected components ($c$)
and isolated vertices ($i$) into types $(c,i)$, we
provided several structural results from which we
established recursive characterizations for many
types in the classification.   More specifically, we
know, by Lemma \ref{lem:BasicsOf(Infty,k)PO}, that
every cograph minimal $(\infty, k )$-polar obstruction
satisfies the constraint $0\le i\le c-1< k+2$, and
we provided recursive characterizations for every
possible type $(c,i)$ but $(c, c-1)$, where
$c \in \{3,\dots,k-1\}$.
Although our results are not enough to describe
all the cograph minimal $(\infty, k)$-polar
obstructions for an arbitrary $k$, we used them to
exhibit complete lists for the cases $k = 2$ and
$k = 3$.   It seems that our techniques might not
be enough to produce an easy recursive formula to
construct the missing obstructions, but we still
think it might be possible to have a general
formula.

\begin{problem}
For a positive integer $k$, find a recursive
characterization for the cograph minimal $(\infty,
k)$-polar obstructions.
\end{problem}

As we observed in Section \ref{sec:remainingTypes},
for some specific values of $c$ and $i$ it can be
proved that there exists exactly one cograph minimal
$(\infty,k)$-polar obstruction of type $(c,i)$.
We believe that our result can be extended in the
following way.

\begin{conjecture}
Let $k,c$ and $i$ be integers such that
$1 \le i \le c-2 \le k$.
Then, there exists exactly one cograph minimal
$(\infty, k)$-polar obstruction of type $(c,i)$.
\end{conjecture}

Additionally, our results on exact lists of
cograph minimal $(\infty, 2)$- and
$(\infty,3)$-polar obstructions support the
following assertion.

\begin{conjecture}
For every cograph minimal $(\infty, k)$-polar
obstruction $G$, the order of $G$ is at most
$3(k+1)$.
\end{conjecture}

A complete solution of problems like the one
addressed in this work might give us a better
understanding of the minimal obstructions for
hereditary properties in the class of cographs,
which in turn, might lead to efficient algorithms
to produce such obstructions directly (as opposed
to making exhaustive searches through all the
cographs).


\end{document}